\documentclass[11pt,a4paper,twoside]{amsart}
\usepackage{amsmath,amssymb,amsfonts,amsthm}
\usepackage{hyperref}
\usepackage{lineno}

\def\K{{\mathcal{K}}}
\def\m{\stackrel{1}{M}}
\def\mm{\stackrel{2}{M}}

\def\K{\stackrel{1}{K}}
\def\KK{\stackrel{2}{K}}
\def\KKK{\stackrel{k}{K}}
\def\st{\stackrel}

\def\l{\lambda}
\def\to{\longrightarrow}
\def\mto{\longmapsto}
\def\a{\alpha}
\def\b{\beta}
\def\e{\eta}
\def\g{\gamma}
\def\p{\phi}
\def\P{\Phi}
\def\G{\Gamma}
\def\s{\psi}
\def\S{\Psi}

\def\eps{\epsilon}
\def\r{\rho}
\def\o{\circ}

\def \dar{\times}
\def \lim{\varprojlim}

\def\N{\mathbb{N}}

\def\F{\mathbb{F}}
\def\R{\mathbb{R}}
\def\E{\mathbb{E}}

\newtheorem{The}{Theorem}[section]
\newtheorem{Pro}[The]{Proposition}
\newtheorem{Lem}[The]{Lemma}
\newtheorem{Cor}[The]{Corollary}
\theoremstyle{definition}
\newtheorem{Def}[The]{Definition}
\newtheorem{Rem}[The]{Remark}
\newtheorem{Examp}[The]{Example}

\begin{document}
\title{Isomorphism classes for higher order tangent bundles}
\author{Ali Suri }
\address{Department  of Mathematics, Faculty of sciences \\
Bu-Ali Sina University, Hamedan 65178, Iran.}
\email{ali.suri@gmail.com \& a.suri@math.iut.ac.ir \& a.suri@basu.ac.ir}
\maketitle {\hspace{2.5cm}}

\begin{abstract}
The tangent bundle $T^kM$ of order $k$,  of a smooth Banach
manifold $M$ consists of all equivalent classes of curves that
agree up to their accelerations of order $k$.
In the previous work of the author he proved that $T^kM$, $1\leq k\leq \infty$, admits a vector bundle structure on $M$ if and only if  $M $ is endowed with a linear connection or equivalently a connection map on $T^kM$ is defined. This bundle structure  depends heavily on the choice of the connection.
In this paper we ask about the extent to which this vector bundle  structure remains isomorphic. To this end we define the notion of the $k$'th order differential $T^kg:T^kM\longrightarrow T^kN$ for a given differentiable map $g$ between manifolds $M$ and $N$. As we shall see, $T^kg$ becomes a vector bundle morphism if the base manifolds are endowed with $g$-related connections. In particular, replacing a connection with a $g$-related one, where $g:M\longrightarrow M$ is a diffeomorphism, follows invariant vector bundle structures.
Finally, using immersions on Hilbert manifolds, convex combination of connection maps and manifold of $C^r$ maps we offer three  examples to support our theory and reveal its interaction  with the known problems such as Sasaki lift of metrics.

\end{abstract}

\pagestyle{headings} \markright{Isomorphism classes for higher order tangent bundles}
\tableofcontents

\section{Introduction}

The tangent bundle   of order k, $T^kM$,  of a smooth  manifold $M$ consists of all equivalent classes of curves that agree up to their accelerations of order k. This bundle is a natural extension of the notion of the usual tangent bundle (see e.g. \cite{Dod-Gal, Miron, Morimoto, Yano}). For example in classical mechanics, $T^kM$ describes the Generalized Particle Mechanics in the autonomous sense \cite{Leon}.

A vector bundle structure for $T^kM$, $2\leq k\leq \infty$, even for $k=2$  is not as evident as in the case of tangent bundle $TM$. In fact it is not always possible to consider $T^kM$ as a vector bundle over $M$ \cite{Dod-Gal, ali, Suri Osck}.

The author in his previous work \cite{Suri Osck} proved that at the presence of a linear connection on $M$ (or equivalently a connection map on $T^kM$), $T^kM$, $2\leq k\leq \infty$,  admits a vector bundle structure on $M$. Moreover it is shown that every linear connection (or equivalently  every Riemannian metric) on $M$,  induces a connection map on $T^kM$.

As an immediate consequence, our suggested  vector bundle structure  allows us
to solve an old problem of differential geometry formulated by Bianchi and
Bompiani  \cite{Miron}, namely the problem of  prolongation
of a Riemannian metric defined on the base manifold $M$ to   $T^kM$,  even for infinite dimensional Hilbertable manifolds \cite{Suri Osck}.

However, as one may have expected, these vector bundle structures depend crucially on the particular connection  chosen \cite{Dod-Gal, Dod-Gal-Vass, ali, Suri Osck}.

In this paper we ask about the
extent of  this vector bundle dependence. We will show that this dependence  is closely related to the notion of \textit{related connection maps} (or conjugate connections) which will be used for a classification of these vector bundle structures.  More precisely we  introduce the higher order differential $T^kg$ of a smooth map $ g:M\to N$ between two manifolds $M$ and $N$ and we investigate when $T^kg$ is linear on fibres. Linearity of $T^k_xg$, $x\in M$, allows us to build a vector bundle morphism $T^kg:T^kM\to T^kN$ (see \cite{ali} and \cite{Dod-Gal-Vass} for the special case $k=2$). As a consequence, we  show that the vector bundle structure on $T^kM$, defined by the aim of a  connection map, remains invariant (isomorphic) if it is replaced by a $g$-related connection map, for some diffeomorphism $g:M\to M$.

If we take one step further by considering $T^\infty M$ and $T^\infty N$, as generalized Fr\'{e}chet vector bundles over $M$ and $N$ respectively (\cite{Suri Osck}), then proving $T^\infty g$ to be a generalized vector bundle morphism becomes much more complicated.
More precisely the set of linear maps between Fr\'{e}chet spaces (the fibre types  of $(T^\infty M,\pi^\infty_M, M)$ and $(T^\infty N,\pi^\infty_N, N)$) does not remain in the category of Fr\'{e}chet spaces \cite{Hamilton, Sch}.
To get around this difficulty, we  employ the projective limit methodology, as in \cite{split, GAl-VB, Suri Osck, Suri Hfbk} etc., to show that $(T^\infty g,g)$  becomes a  generalized vector bundle morphism at the presence of $g$-related connections on $M$ and $N$.

Afterward, as an application, we  settle our results to the special case of $f:M\to N$, where $f$ is an immersion and $N$ is a Riemannian Hilbert manifold. As a result, this special case tells us that higher order differential of an isometry   is again an isometry with respect to the induced Sasaki-type metrics.

Then we check the vector bundle dependence on convex combination  of connection maps.

We close this article by addressing, in part, the case of
manifold of $C^r$ maps  between manifolds $M$ and $N$ denoted by $C^r(N,M)$.

Through this paper all the maps and manifolds are assumed to be smooth,
but, except  in section \ref{section T infty g}, a lesser degree of differentiability can be assumed.

The readers who are unfamiliar with infinite dimensional manifolds and spaces, can easily replace the model spaces with  Euclidean spaces.

Most of the results of this paper are novel even for the case of finite dimensional manifolds.
%
%
\section{Preliminaries}\label{Sec preliminaries}
In this section we summarize  the necessary preliminary materials that we need for a
self contained presentation of our paper and we record our notation.

At various points in this article, we  wish to have an  explicit formula for
higher order differentials of  compositions  of smooth functions. Hence, we begin with a short description of higher order chain rule. Let $f:\mathbb{I}\subseteq\R\to U$ and $g:U\to V$ be $k$-times Fr\'{e}chet differentiable where $U$ and $V$ are open subsets of the Banach spaces $\E$ and $\E'$ respectively. Then it is known that $g\o f$ is $k$-times Fr\'{e}chet differentiable and
\begin{eqnarray}\label{higher order chain rule}
&&(g\o f)^{(k)}(0)=\\
\nonumber &&\sum \frac{k!}{l_1!\dots l_i !m_1!\dots m_k!}d^ig(f(0))\big(f^{(l_1)}(0),\dots,f^{(l_i)}(0) \big)
\end{eqnarray}
where the second sum is over all ordered $i$-tuples $(l_1,\dots,l_i)$ of integers $l_1,\dots,l_i$
such that $l_1+\dots+l_i=k$ and  $1\leq l_1\leq\dots\leq l_i\leq k$ with $i$ varying form $1$ to $k$. Moreover for
any $j\in\{1,\dots k\}$, $m_j$ is the number of $l_1,\dots, l_i$ equal to $j$  (\cite{Averbuh}, p. 234, \cite{Lloyd}, p. 359 or \cite{Schwartz}, p. 262).
The coefficient $\frac{k!}{l_1!\dots l_i! m_1!\dots m_k!}$ will henceforth be denoted by $a_{(l_1,\dots,l_i)}^k.$

We proceed  with a short description of  infinite-dimensional manifolds
and their tangent bundles.
Let M be a manifold modeled on
the Banach space $\E$. For any $x_0\in M$ define
\begin{equation*}
C_{x_0}:=\{\g:(-\eps,\eps)\to M~;~ \g(0)=x_0  ~ \textrm{and}~ \g
~\textrm{is smooth }\}.
\end{equation*}
As a natural extension of the  tangent bundle $TM$ define the
following equivalence relation. For $\g\in C_{x_0}$, set $\g^{(1)}(t)=\g'(t)$ and  $\g^{(k)}(t)={\g^{(k-1)}}'(t)$ where $k\in\N$ and $k\geq 2$. Two curves  $\g_1,\g_2 \in C_{x_0}$
are said to be $k$-equivalent, denoted by $\g_1\approx_{x_0}^k\g_2$,
if and only if $\g_1^{(j)}(0)=\g_2^{(j)}(0)$  for all $1\leq j\leq k$.
Define ${T}^k_{x_0}M:=C_{x_0}/\approx_{x_0}^k$ and the \textbf{
tangent bundle of order  $k$ or $k$-osculating bundle} of M to be
${T}^kM:=\bigcup_{x\in M}{T}^k_{x}M$. We denote by $[\g,{x_0}]_k$ the
representative of the equivalence class containing $\g$  and define
the canonical projection ${\pi}_M^k:T^kM\to M$ which projects
$[\g,{x_0}]_k$ onto $ x_0$.

Let $\mathcal{A}=\{(U_\a,\phi_\a)\}_{\a\in I}$ be a $C^\infty$ atlas
for $M$. For any $\a\in I$ define
\begin{eqnarray*}
\phi_\a^k:{\pi_M^k}^{-1}(U_\a)&\to& \p_\a(U_\a)\times\E^k\\
{[\gamma,{x_0}]_k}&\mto& \big( (\p_\a
\circ\g)(0),(\p_\a\circ\g)'(0),...,\frac{1}{k!}(\p_\a\circ\g)^{(k)}(0)\big)
\end{eqnarray*}
%
%
\begin{The}\label{fibre bundle structure for TkM}
 The family
$\mathcal{A}_k=\{({\pi_M^k}^{-1}(U_\a),\phi_\a^k)\}_{\a\in I}$ declares
a smooth fibre bundle (not generally a vector bundle)  structure for $T^kM$ over $M$ \cite{Suri Osck}.
\end{The}
%
%
Consider the $C^\infty(T^kM)$-linear map $\mathbb{J}  :\mathfrak{X}(T^kM)\to \mathfrak{X}(T^kM)$  s.t. locally on a chart $(\p_\a^k,{\pi_M^k}^{-1}(U_\a))$, $\mathbb{J}$ is given by\\
$$\mathbb{J}_{\a}(u;y,\e_1,\dots,\e_k)=(u;0,y,\e_1,\dots,\e_{k-1}).$$
for any $u:=(x,\xi_1,\dots,\xi_k)\in T^kM$ and every
$(u;y,\e_1,\dots,\e_k)\in T_uT^kM$.
\begin{Def}\label{Def Connection map}
A connection map on ${T}^kM$
is a vector bundle morphism  
$$K=(\K ,\KK...,\KKK):T{T}^kM\to\Big(\oplus_{i=1}^kTM,\oplus_{i=1}^k\pi_M^1,\oplus_{i=1}^kM\Big)\vspace{-2mm}$$
such that for any $1\leq a\leq k-1$, $\KKK\o
\mathbb{J}^a=\stackrel{k-a}{K}$ and $\KKK\o \mathbb{J}^k={\pi_M^k}_*$ \cite{Ioan, Suri Osck}.
\end{Def}
In order to carry out the local structure  of a connection map, we state the following lemma according to \cite{Suri Osck}.
\begin{Lem}\label{local form of a connection map}  Locally on a chart $({\pi^k_M}^{-1}(U_\a),\phi_\a^k)$
there are smooth maps $\stackrel{i}{M}_\a:U_\a\dar\E^k\to L(\E,\E)$, $1\leq i\leq k$, such that the connection map
$$K_{\a}:=\oplus_{i=1}^k\phi_\a^1 \o K \o T{\phi_\a^k}^{-1}=(\K_\a,...,\KKK_\a )$$ at $(u;y,\e_1,...,\e_k)\in
T_uT^kM$ is given by
\begin{eqnarray}\label{local form of K}
&&K|_{U_\a}(u;y,\e_1,...,\e_k)=\\
\nonumber&&\bigoplus_{i=1}^k\Big(x,\e_i+\m_\a(u)\e_{i-1}+
\mm_\a(u)\e_{i-2}+...+\stackrel{i}{M}_\a(u) y\Big).
\end{eqnarray}
\end{Lem}
On the common area of the charts $({\pi^k_M}^{-1}(U_\a),\p_\a^k)$ and $({\pi^k_N}^{-1}(U_\b),\p_\b^k)$,  the local components $\stackrel{i}{M}_\a$ and $\stackrel{i}{M}_\b$ are related via the transformation rule given by remark \ref{Rem transformation rule for Mi a and M i b}.
%
%
%
%
%
%
%

Let $M$ and $N$ be two smooth manifolds modeled on the Banach spaces $\E$ and $\E'$. Motivated by \cite{Vass}, \cite{Dod-Gal-Vass} and \cite{ali} we state the following two definitions.
\begin{Def}\label{Definition of T^kg}
Let $g:M\to N$ be a smooth map. For $k\in\N$ define the $k$'th order differential of $g$ by
\begin{eqnarray*}
T^kg:T^kM &\to & T^kN\\
{[\g,x]}_k &\mto& [g\o\g,g(x)]_k
\end{eqnarray*}
\end{Def}
To show that the above definition is well defined consider another  representative $[\delta,x]_k$ of the class $[\g,x]_k\in T^k_xM$.
Then, for any integer $1\leq i\leq k$, using (\ref{higher order chain rule}) we have
\begin{eqnarray*}
(g\o\g)^{(i)}(0)&=&\sum a_{(l_1,\dots,l_j)}^k d^jg(x)\big(\g^{(l_1)}(0),\dots,\g^{(l_j)}(0) \big)\\
&=&\sum a_{(l_1,\dots,l_j)}^k d^jg(x)\big(\delta^{(l_1)}(0),\dots,\delta^{(l_j)}(0) \big)\\
&=&(g\o\delta)^{(i)}(0)
\end{eqnarray*}
that is $T^kg$ is well defined.

\begin{Def}
Let $K_M$ and $K_N$ be two connection maps on $M$ and $N$ respectively and $g:M\to N$ be a smooth map.
$K_M$ and $K_N$ are called $g$-related if they commute with the differentials of $g$ in the following manner
\begin{equation}\label{g related connections equation}
K_N\o T T^kg=\oplus_{i=1}^k Tg\o K_M.
\end{equation}
\end{Def}
\begin{Rem}
If $k=1$ then, the above definition agrees with that of \cite{Vass}, \cite{Dod-Gal-Vass} and \cite{ali}.
\end{Rem}

From now on fix the atlas $\mathcal{B}=\{(V_\b,\s_\b)\}_{\b\in J}$ for  $N$
and construct the proposed   atlas discussed in theorem \ref{fibre bundle structure for TkM} for  $T^kN$ which we  denote  by $\mathcal{B}_k=\{({\pi_N^k}^{-1}(V_\b),\psi_\b^k)\}_{\b\in J}$. The model space of $N$ is the Banach space $\E'$.

For suitably chosen  charts $({\pi_M^k}^{-1}(U_\a),\p_\a^k)$ and $({\pi_N^k}^{-1}(V_\b),\s_\b^k)$, of $M$ and $N$ respectively, we have
\begin{eqnarray*}
&&(\oplus_{i=1}^k\s_\b^1)\o K_N\o T T^kg\o T{\phi_\a^k}^{-1}=\\
&&=(\oplus_{i=1}^k\s_\b^1)\o K_N \o {T\s_\b^k}^{-1}\o T\psi_\b^k\o T T^kg\o T{\phi_\a^k}^{-1}\\
&&:=K_{N,\b}\o T T^kg_{\b\a}
\end{eqnarray*}
and
\begin{equation*}
(\oplus_{i=1}^k\s_\b^1)\o (\oplus_{i=1}^k Tg)\o K_M\o  T{\phi_\a^k}^{-1}:=(\oplus_{i=1}^k Tg_{\b\a})\o K_{M,\a}
\end{equation*}
where $g_{\b\a}:=\s_\b\o g\o\p_\a^{-1}$ and $K_{M,\a}$ and $K_{N,\b}$ stand for the local representations of the connection
maps pointed out by lemma \ref{local form of a connection map}.

In order to reveal the local compatibility condition for  $g$-related   connections, for any $(x,\xi_1,\dots,\xi_k,y,\e_1,\dots,\e_k)\in U_\a\times \E^{2k+1}$, we define the following auxiliary
curve.
\begin{eqnarray*}
\bar{c}:(-\eps,\eps)^2&\to &\p_\a(U_\a)\subseteq\E\\
(t,s)&\mto& x+sy+\sum_{i=1}^kt^i(\xi_i+s\e_i)
\end{eqnarray*}
Now, evaluating $(\oplus_{i=1}^k Tg_{\b\a})\o K_{M,\a}$ at $(u,y,\e_1,\dots,\e_k)$ yields
\begin{eqnarray*}
&&(\oplus_{i=1}^k Tg_{\b\a})\o K_{M,\a}(u,y,\e_1,\dots,\e_k)=\\
&&\bigoplus_{i=1}^k\Big(g_{\b\a}(x),dg_{\b\a}(x)[\e_i+\m_\a(u)\e_{i-1}+
\mm_\a(u)\e_{i-2}+...+\stackrel{i}{M}_\a(u) y]\Big).
\end{eqnarray*}
On the other hand
\begin{eqnarray*}
&&K_{N,\b}\o T T^kg_{\b\a}(u,y,\e_1,\dots,\e_k)\\
%
%
&&=\bigoplus_{i=1}^k\Big(g_{\b\a}(x), \bar{\e}_i +\st{1}{N}_\b(\bar{u})
\bar{\e}_{i-1}+\dots+\st{i}{N}_\b(\bar{u})\bar{y}         \Big)
\end{eqnarray*}
where $\bar\g(t):=\bar{c}(t,0)$, $\bar{u}=((g_{\a\b}\o\bar\g)(0),\dots,\frac{1}{k!}(g_{\a\b}\o\bar\g)^{(k)}(0))$, $\bar{y}=\frac{\partial }{\partial s}(g_{\b\a}\o\bar{c})(0,0)$,
$\bar{\e}_i=\frac{1}{i!}\frac{\partial^{i+1}}{\partial s\partial t^i}(g_{\b\a}\o\bar{c})(0,0)$ and $\st{i}{N}_\b$ are the local components of $K_N$
for $1\leq i\leq k$.

With the notation as above, we have the following \textbf{important} compatibility condition which locally declares  $g$-related connection maps
\begin{eqnarray}\label{g related local compatiblity condition}
&&dg_{\b\a}(x)[\e_i+\m_\a(u)\e_{i-1}+
\mm_\a(u)\e_{i-2}+...+\stackrel{i}{M}_\a(u) y]\\
\nonumber&&=\bar{\e}_i +\st{1}{N}_\b(\bar{u})
\bar{\e}_{i-1}+\dots+\st{i}{N}_\b(\bar{u})\bar{y}; ~~~~1\leq i\leq k.
\end{eqnarray}

\begin{Rem}\label{Rem compatibility condition for connections when k=1}
If  $k=1$, then the last equation coincides with the local compatibility condition for  $g$-related connections $K_M$ and $K_N$ on $M$ and $N$ as it is stated  in \cite{Vass} p. 299.
\end{Rem}
\begin{Rem}\label{Rem transformation rule for Mi a and M i b}
Whenever $M=N$, $K_M=K_N:=K$ and  $g=id_M$, the equation (\ref{g related local compatiblity condition}) reduces to the compatibility condition which locally the connection maps  on common charts must satisfy \cite{Suri Osck}.
\end{Rem}

%
%
In what follows, we determine a canonical connection map on $T^kM$ depending only on a given linear connection (Riemannian metric) on the base manifold $M$.
Keeping the formalisms of  \cite{ali-iejgeo, Vass, Vilms} we state the following proposition according to \cite{Suri Osck}.
\begin{Pro}\label{lifted connection to osckM}
Let $\nabla$ be  a linear connection on $M$ with the local
components (Christoffel symbols) $\{\G\}_{\a\in I}$. There exists an induced connection
map on ${T}^kM$ with the following local components.
\begin{eqnarray*}
&&\m_\a(x,\xi_1)y=\Gamma_\a(x,\xi_1)y\\
&&\mm_\a(x,\xi_1,\xi_2)y=\frac{1}{2}\Big(\sum_{i=1}^2\partial_i\m_\a(x,\xi_1)(y,i\xi_i)+\m_\a(x,\xi_1)[\m_\a(x,\xi_1)y]\Big),\\
&&\vdots\\
&&\st{k}{M}_\a(x,\xi_1,...,\xi_k)y=\frac{1}{k}\Big(\sum_{i=1}^{k}\partial_i\st{k-1}{M}_\a(x,\xi_1,...,\xi_{k-1})(y,i\xi_i)\\
&&\hspace{3.7cm}+\m_\a(x,\xi_1)[\st{k-1}{M}_\a(x,\xi_1,...,\xi_{k-1})y]\Big).
\end{eqnarray*}
\end{Pro}
\subsection{Lifting of related linear connections to higher order tangent bundles}

In this section we show that lift of $g$-related linear connections remain $g$-related. More precisely let $g:M\to N$ be a smooth map between differentiable  manifolds $M$ and $N$ and  $\nabla_M$ and $\nabla_N$ be two  linear connections on $M$ and $N$ respectively. Moreover suppose that  $K_M$ and $K_N$ be  the lifted connection maps (as in proposition  \ref{lifted connection to osckM}) generated by $\nabla_M$ and $\nabla_N$ on $T^kM$ and $T^kN$ respectively.
\begin{Pro}\label{theorem lifts of related connections remain related}
If  $\nabla_M$ and $\nabla_N$ are  $g$-related,   then $K_M$ and $K_N$ are $g$-related connection maps too.
\end{Pro}
\begin{proof}
We shall prove that, for  any $\a\in I$, $\b\in J$ with $g(U_\a)\subseteq V_\b$ ,  the local components $\{\st{i}{M}_\a\}_{i=1,\dots,k}$ and   $\{\st{i}{N}_\b\}_{i=1,\dots,k}$
from proposition  \ref{lifted connection to osckM} satisfy  the condition (\ref{g related local compatiblity condition}).
The proof is by induction on $i$.

For the base step of induction consider  the Christoffel symbols $\{\G_\a^M\}_{\a\in I}$ and $\{\G_\b^N\}_{\b\in J}$ of the $g$ related connections $\nabla_M$ and $\nabla_N$. The compatibility condition for $\G_\a^M$ and $\G_\b^N$ is given by
\begin{eqnarray*}
dg_{\b\a}(x)[\eta_1+\G_\a^M(x,\xi_1)y]&=&dg_{\b\a}(x)(\e_1)+d^2g_{\b\a}(x)(\xi_1,y)    \\
&&+\G_\b^N\big(g_{\b\a}(x),dg_{\b\a}(x)(\xi_1)\big)dg_{\b\a}(x)(y)
\end{eqnarray*}
(for more details see  \cite{Dod-Gal-Vass, ali, Vass}). Moreover we  have $\st{1}{M}_\a=\G_\a^M$ and $\st{1}{N}_\b=\G_\b^N$. These last three equalities show that $\st{1}{M}_\a$ and $\st{1}{N}_\b$ satisfy  (\ref{g related local compatiblity condition}).

By induction, we assume that for $i=1,\dots k-1$, $\st{i}{M}_\a$ and $\st{i}{N}_\b$; the rule (\ref{g related local compatiblity condition}) is verified i.e.
\begin{eqnarray}
&&dg_{\b\a}(x)[\e_i+\m_\a(u_1)\e_{i-1}+
\mm_\a(u_2)\e_{i-2}+...+\stackrel{i}{M}_\a(u_i) y]\\
\nonumber&&=\bar{\e}_i +\st{1}{N}_\b(\bar{u}_1)
\bar{\e}_{i-1}+\dots+\st{i}{N}_\b(\bar{u}_i)\bar{y}
\end{eqnarray}
where $u_j:=(x,\xi_1,\dots,\xi_j)$, $\bar{u}_j=\big((g_{\b\a}\o\bar\g)(0),\dots,\frac{1}{j!}(g_{\b\a}\o\bar\g)^{(j)}(0)\big)$,
$\bar{y}=\frac{\partial}{\partial s}(g_{\b\a}\o \bar{c})(0,0)$, $\bar{\e}_j=\frac{1}{j!}\frac{\partial^{j+1}}{\partial s\partial t^j}(g_{\b\a}\o \bar{c})(0,0)$, for $1\leq j\leq i$, and $\bar{c}$ and $\bar \g$ are  as in   (\ref{g related local compatiblity condition}).
Then from the definition of $\st{k}{M}_\a$ at $(u_k;y,0,\dots0)$  we get
\begin{eqnarray*}
kdg_{\b\a}(x)\st{k}{M}_\a(u_k)y&=&dg_{\b\a}(x)\big\{
\sum_{i=1}^k\partial_i\st{k-1}{M}_\a(u_{k-1})(y,i\xi_i)\\
&&+\st{1}{M}_\a(x,\xi_1)\st{k-1}{M}_\a(u_{k-1})y\big\}
\end{eqnarray*}
It perhaps worth remarking that
\begin{eqnarray*}
&&dg_{\b\a}(x)\{\partial_1\st{k-1}{M}_\a(u_{k-1})(y,\xi_1)\}\\
&=&dg_{\b\a}(x)\frac{d}{ds}   \big(   \underbrace{\st{k-1}{M}_\a(x+s\xi_1,\xi_1,\dots,\xi_{k-1})y}_{\star}   \big)\\
&=&\textrm{lim}_{s\rightarrow 0}      \Big(  dg_{\b\a}(x)\star  - dg_{\b\a}(x)\st{k-1}{M}_\a(u_{k-1})y
\pm  dg_{\b\a}(x+s\xi_1)\star   \Big)/{s}\\
&=& -\textrm{lim}_{s\rightarrow 0}   \Big( dg_{\b\a}(x+s\xi_1)\star- dg_{\b\a}(x)\star         \Big)/{s}\\
&&+ \textrm{lim}_{s\rightarrow 0} \Big(   \ dg_{\b\a}(x+s\xi_1)\star   - dg_{\b\a}(x)\st{k-1}{M}_\a(u_{k-1})y\Big)/{s}.\\
&=&-d^2g_{\b\a}(x) \big(  \xi_1,\st{k-1}{M}_\a(u_{k-1})y \big) \\
&&+ \textrm{lim}_{s\rightarrow 0} \Big(   \ dg_{\b\a}(x+s\xi_1)\star   - dg_{\b\a}(x)\st{k-1}{M}_\a(u_{k-1})y\Big)/{s}.\\
\end{eqnarray*}
Now, by the induction hypothesis   we have
\begin{eqnarray*}
&&dg_{\b\a}(x)\{\partial_1\st{k-1}{M}_\a(u_{k-1})(y,\xi_1)\}
=-d^2g_{\b\a}(x)(\xi_1,\st{k-1}{M}_\a(u_{k-1})y\big)\\
&&+\frac{\partial}{\partial h}\Big\{\frac{\partial ^k}{(k-1)!\partial s\partial t^{k-1}}(g_{\b\a}\o \bar{c}_1)(t,s,h)\\
&&+\st{1}{N}_\b[ (g_{\b\a}\o \bar{\g}_1)(t,h),\frac{\partial}{\partial t}(g_{\b\a}\o \bar{\g}_1)(t,h)]\frac{\partial ^{k-1}}{(k-2)!\partial s\partial t^{k-2}}(g_{\b\a}\o \bar{c}_1)(t,s,h)\\
&&+\dots+ \st{k-1}{N}_\b[ (g_{\b\a}\o \bar{\g}_1)(t,h),\frac{\partial}{\partial t}(g_{\b\a}\o \bar{\g}_1)(t,h),\dots\\
&&,\frac{\partial^{k-1}}{(k-1)!\partial t^{k-1}}(g_{\b\a}\o \bar{\g}_1)(t,h)]\frac{\partial }{\partial s}(g_{\b\a}\o \bar{c}_1)(t,s,h)\Big\}|_{t=s=h=0}
\end{eqnarray*}
and for $2\leq i\leq k$
\begin{eqnarray*}
&&dg_{\b\a}(x)\{\partial_i\st{k-1}{M}_\a(u_{k-1})(y,i\xi_i)\}
=\frac{\partial}{\partial h}\Big\{\frac{\partial ^k}{(k-1)!\partial s\partial t^{k-1}}(g_{\b\a}\o \bar{c}_i)(t,s,h)\\
&&+\st{1}{N}_\b[ (g_{\b\a}\o \bar{\g}_i)(t,h),\frac{\partial}{\partial t}(g_{\b\a}\o \bar{\g}_i)(t,h)]\frac{\partial ^{k-1}}{(k-2)!\partial s\partial t^{k-2}}(g_{\b\a}\o \bar{c}_i)(t,s,h)\\
&&+\dots+ \st{k-1}{N}_\b[ (g_{\b\a}\o \bar{\g}_i)(t,h),\frac{\partial}{\partial t}(g_{\b\a}\o \bar{\g}_i)(t,h),\dots\\
&&,\frac{\partial^{k-1}}{(k-1)!\partial t^{k-1}}(g_{\b\a}\o \bar{\g}_i)(t,h)]\frac{\partial }{\partial s}(g_{\b\a}\o \bar{c}_i)(t,s,h)\Big\}|_{t=s=h=0}
\end{eqnarray*}
where $\bar{c}_1(t,s,h):=x+h\xi_i+s y+t\xi_1+\dots +t^k\xi_k,$ $\bar\g_1(t,h)=\bar{c}_1(t,0,h),\dots$,
$$\bar{c}_i(t,s,h):=x+sy +t\xi_1+\dots +t^{i-2}\xi_{i-2}+t^{i-1}(\xi_{i-1}+hi\xi_i)+t^i\xi_i+\dots +t^k\xi_k$$
and $\bar\g_i(t,h)=\bar{c}_i(t,0,h)$.
As a result
\begin{eqnarray*}
&&kdg_{\b\a}(x)\st{k}{M}_\a(u_k)y\\
&=&-d^2g_{\b\a}(x)(\xi_1,\st{k-1}{M}_\a(u_{k-1})y\big)+\sum_{i=1}^k\frac{\partial}{\partial h}\Big\{\frac{\partial ^k}{(k-1)!\partial s\partial t^{k-1}}(g_{\b\a}\o \bar{c}_i)(t,s,h)\\
&&+\st{1}{N}_\b[ (g_{\b\a}\o \bar{\g}_i)(t,h),\frac{\partial}{\partial t}(g_{\b\a}\o \bar{\g}_i)(t,h)]\frac{\partial ^{k-1}}{(k-2)!\partial s\partial t^{k-2}}(g_{\b\a}\o \bar{c}_i)(t,s,h)\\
&&+\dots+ \st{k-1}{N}_\b[ (g_{\b\a}\o \bar{\g}_i)(t,h),\frac{\partial}{\partial t}(g_{\b\a}\o \bar{\g}_i)(t,h),\dots\\
&&,\frac{\partial^{k-1}}{(k-1)!\partial t^{k-1}}(g_{\b\a}\o \bar{\g}_i)(t,h)]\frac{\partial }{\partial s}(g_{\b\a}\o \bar{c}_i)(t,s,h)\Big\}|_{t=s=h=0}\\
&& +\st{1}{N}_\b(\bar{u}_1)[dg_{\b\a}(x)\st{k-1}{M}_\a({u}_{k-1}){y}]+ d^2g_{\b\a}(x)(\xi_1,\st{k-1}{M}_\a(u_{k-1})y\big).
\end{eqnarray*}
To simplify the above terms, again we use the induction hypothesis for the last line and we get
\begin{eqnarray*}
&&kdg_{\b\a}(x)\st{k}{M}_\a(u_k)y\\
&&= \frac{\partial ^{k+1}}{(k-1)!\partial h\partial s\partial t^{k-1}}\sum_{i=1}^{k}(g_{\b\a}\o \bar{c}_i)(t,s,h)|_{t=s=h=0}\\
&&+ \sum_{j=1}^2\partial _j\st{1}{N}_\b(\bar{u}_1)(\bar{\e}_{k-2},   \sum_{i=1}^k\{\frac{\partial^{j}}{(j-1)!\partial h\partial t^{j-1}}(g_{\b\a}\o \bar{\g}_i)(t,h)\})+\dots\\
&&+\sum_{j=1}^k\partial_j\st{k-1}{N}_\b(\bar{u}_{k-1})(\bar{y},   \sum_{i=1}^k\{\frac{\partial^{j}}{(j-1)!\partial h\partial t^{j-1}}(g_{\b\a}\o \bar{\g}_i)(t,h)\})|_{t=h=0}\\
\end{eqnarray*}
\begin{eqnarray*}
&& +\st{1}{N}_\b(\bar{u}_1)(\frac{\partial^k}{(k-2)!\partial h\partial s\partial t^{k-2}}\sum_{i=1}^k(g_{\b\a}\o \bar{c}_i)(t,s,h)|_{t=s=h=0})
+\dots\\
&& +\st{k-1}{N}_\b(\bar{u}_{k-1})(\frac{\partial^2}{\partial h\partial s}\sum_{i=1}^k(g_{\b\a}\o \bar{c}_i)(t,s,h)|_{t=s=h=0})
\\
%
%
&&+\st{1}{N}_\b(\bar{u}_1)[\bar{\e}_{k-1}+\st{1}{N}_\b(\bar{u}_1)\bar{\e}_{k-2}+\dots+\st{k-1}{N}_\b(\bar{u}_{k-1})\bar{y}]
\end{eqnarray*}
Now, using lemma \ref{Lem 2}, it follows that $kdg_{\b\a}(x)\st{k}{M}_\a(u_k)y$ is equal to
\begin{eqnarray*}
&&k\bar{\e}_k+
\sum_{j=1}^2\partial _j\st{1}{N}_\b(\bar{u}_1)(\bar{\e}_{k-2},j\bar{\xi}_j) +\dots+\sum_{j=1}^k\partial_j\st{k-1}{N}_\b(\bar{u}_{k-1})(\bar{y},j\bar{\xi}_j)\\
&&+(k-1)\st{1}{N}_\b(\bar{u}_1)\bar{\e}_{k-1} + \dots + \st{k-1}{N}_\b(\bar{u}_{k-1})\bar{\e}_{1}\\
&& +\st{1}{N}_\b(\bar{u}_1)[\bar{\e}_{k-1}+\st{1}{N}_\b(\bar{u}_1)\bar{\e}_{k-2}+\dots+\st{k-1}{N}_\b(\bar{u}_{k-1})\bar{y}]
\end{eqnarray*}
\begin{eqnarray*}
&=&k\bar{\e}_k+2\st{2}{N}_\b(\bar{u}_{2})\bar{\e}_{k-2}+\dots+k\st{k}{N}_\b(\bar{u}_{k})\bar{y}+(k-1)\st{1}{N}_\b(\bar{u}_1)\bar{\e}_{k-1} \\
&&+(k-2)\st{2}{N}_\b(\bar{u}_2)\bar{\e}_{k-2} +\dots + \st{k-1}{N}_\b(\bar{u}_{k-1})\bar{\e}_{1} +\st{1}{N}_\b(\bar{u}_1)(\bar{\e}_{k-1})\\
&=&k\bar{\e}_k+ (1 +(k-1))\st{1}{N}_\b(\bar{u}_1)\bar{\e}_{k-1} +\dots +((k-1)+1)\st{k-1}{N}_\b(\bar{u}_{k-1})\bar{\e}_{1} \\
&&+k\st{k}{N}_\b(\bar{u}_{k})\bar{y}\\
&=&k   \Big(    \bar{\e}_k+\st{1}{N}_\b(\bar{u}_1)\bar{\e}_{k-1} +\dots+\st{k}{N}_\b(\bar{u}_{k})\bar{y}  \Big)
\end{eqnarray*}
where $\bar{\xi}_j=Proj_{j+1}(\bar{u}_{k})$ and $Proj_{j+1}$, $0\leq j\leq k-1$, is the projection map to the $(j+1)$'th factor. As a consequence we obtain
\begin{eqnarray*}
dg_{\b\a}(x)\st{k}{M}_\a(x,\xi_1,\dots,\xi_k)y=\bar{\e}_k+\st{1}{N}_\b(\bar{u}_1)\bar{\e}_{k-1} +\dots+\st{k}{N}_\b(\bar{u}_{k})\bar{y}.
\end{eqnarray*}

However, as we are about to see, the desired  compatibility condition  is finally handled. More precisely we have shown that for the  vector bundle morphisms $K_M$ and $K_N$ locally
\begin{eqnarray*}
\oplus_{i=1}^k Tg_{\b\a}(x)K_{M,\a}(u;y,0,\dots,0)=K_{N,\b}\o TT^kg_{\b\a}(u;y,0,\dots,0).
\end{eqnarray*}
Moreover by the induction hypothesis we have
\begin{eqnarray*}
\oplus_{i=1}^k Tg_{\b\a}(x)K_{M,\a}(u;0,\e_1,\e_2,\dots,\e_k)=K_{N,\b}\o TT^kg_{\b\a}(u;0,\e_1,\e_2,\dots,\e_k).
\end{eqnarray*}
Since $K_M$ and $K_N$ are vector bundle morphisms (and linear on fibres) we can add the last two equalities. As a consequence we have
\begin{eqnarray*}
\oplus_{i=1}^k Tg_{\b\a}(x)K_{M,\a}(u;y,\e_1,\e_2,\dots,\e_k)=K_{N,\b}\o TT^kg_{\b\a}(u;y,\e_1,\e_2,\dots,\e_k).
\end{eqnarray*}
This last means that  $K_M$ and $K_N$ are $g$-related connection maps on the $k$'th order tangent bundles of $M$ and $N$ respectively, constructed only with the help of the connections $\nabla_M$ and $\nabla_N$.
\end{proof}
%


%
%
\subsection{$T^kM$ as a vector bundle}
For $k\geq 2$, the bundle structure defined in theorem \ref{fibre
bundle structure for TkM} is quite far from being a vector bundle due to
the complicated nonlinear transition functions. However, according to \cite{Suri Osck} we have the following main theorem.
\begin{The}\label{osc k admits a vb}
Let $\nabla$ be a linear connection on $M$ and $K$ the induced
connection map introduced in proposition  \ref{lifted connection to
osckM}. The following trivializations  define a vector bundle
structure on $\pi^k_M:{T}^kM\to M$ with the structure group
$GL(\E^k)$.
\begin{eqnarray}\label{trivializations for TkM as a vb}
\P_\a^k:{\pi^k_M}^{-1}(U_\a)&\to& \p_\a(U_\a)\times \E^k\\
\nonumber {[\g,x]}_k&\mto&(\g_\a(0),\g_\a^{(1)}(0),z^2([\g,x]_k),\dots,z^k([\g,x]_k))
\end{eqnarray}
where $\g_\a=\p_\a\o\g$ and
\begin{eqnarray*}
z^2_\a([\g,x]_k)&=&\frac{1}{2}\Big\{\frac{\g_\a^{(2)}(0)}{1!}+\m_\a[\g_\a(0),\g_\a^{(1)}(0)]\g_\a^{(1)}(0)\Big\},\\
\vdots\\
%
%
%
z^k_\a([\g,x]_k)&=&\frac{1}{k}\Big\{\frac{\g_\a^{(k)}(0)}{(k-1)!}+\m_\a[\g_\a(0),\g_\a^{(1)}(0)]\frac{\g_\a^{(k-1)}(0)}{(k-2)!}+\dots\\
&&+\stackrel{k-1}{M}_\a
[\g_\a(0),\g_\a^{(1)}(0),\dots,\frac{\g_\a^{(k-1)}(0)}{(k-1)!}]\g_\a^{(1)}(0)\Big\}.
\end{eqnarray*}
Moreover setting $\P_{\b\a}^k=\P_\b^k\o{{\P}_\a^k}^{-1}$, $\p_{\b\a}=\p_\b\o\p_\a^{-1}$ and $U_{\b\a}=U_\a\cap U_\b$  the transition map
\begin{equation*}
{{\P}_{\b\a}^k}:U_{\b\a}\to GL(\E^k)
\end{equation*}
is given by $
{{\P}_{\b\a}^k}(x)(\xi_1,\xi_2,\dots,\xi_k)=\Big(\p_{\b\a}(x),d\p_{\b\a}(x)\xi_1,...,d\p_{\b\a}(x)\xi_k\Big)$ that is $T^kM$, as a vector bundle, is isomorphic to $\oplus_{i=1}^kTM$.
\end{The}
The converse of the above theorem is also true i.e.  if $\pi^k_M:T^kM\to M$, for some $k\geq 2$, admits a vector bundle
structure isomorphic to $\oplus_{i=1}^kTM$, then  a linear connection on $M$ can be defined \cite{Suri Osck}.
\begin{Rem}
One can replace the induced connection map in theorem \ref{osc k admits a vb} with a general connection map (in the sense of definition \ref{Def Connection map}) and prove the theorem in a similar fashion.
One reason for using this rather elaborate model (induced connection maps)  is that it permits a  concrete way of constructing.
\end{Rem}
\begin{Rem}
For $i< k$
\begin{eqnarray*}
\pi_M^{k,i}:T^kM&\to &T^iM\\
{[\g,x]}_k&\mto& {[\g,x]}_i
\end{eqnarray*}
also admits a vector bundle structure \cite{Suri Osck}.
\end{Rem}
%
%
%
%
\begin{Rem}
If the base manifold $M$ is $C^k$-partitionable, then the existence of  a connection on $M$  and equivalently  a vector bundle structure on $\pi^k_M:T^kM\to M$ is guaranteed (see also \cite{ali} p. 94).

However,  existence of a $C^k$
partition of unity on the discussing manifolds puts some restrictions on the model spaces.
According to \cite{Kunak} if $M$ is a paracompact $C^k$ manifold
modeled on a separable $SC^p$ Banach space with $p\geq
max\{2,k\}$, then $M$ admits a $C^k$ partition of unity. As a
corollary  every paracompact $C^k$  manifold modeled on a
separable Hilbert space admits a $C^k$ partition of unity
\cite{ Kling, Kunak, Lang}.
\end{Rem}
As we have shown, the vector bundle structure on $T^kM$ depends heavily on the chosen  linear connection (see e.g. example \ref{Example convex combination}).
In the next section we ask about the extent to which this vector bundle  structure remains isomorphic.
%
%
%

\section{$T^kg $ as a vector bundle morphism}\label{section Tkg as a vb morphism}
For a differentiable map $g:M\to N$, in contrast to  $T^1g=Tg:TM\to TN$, the tangent map $T^kg$, even for $k=2$, is  not necessarily
a vector bundle morphism \cite{Dod-Gal-Vass, ali}. In this section first, we investigate under what conditions $T^kg$ becomes a vector bundle morphism.

Let $K_M$ and $K_N$ be two connection maps on $T^kM$ and $T^kN$ respectively, possibly, induced by linear connection $\nabla_M$ and $\nabla_N$.
For $[\g,x]_k\in T^kM$ suppose that $({\pi_M^k}^{-1}(U_\a),\Phi_\a^k)$ and $( {\pi_N^k}^{-1}(V_\b),\Psi_\b^k)$
be two vector bundle  trivializations (given by theorem \ref{osc k admits a vb}) around $x\in M$ and $g(x)\in N$ respectively.  \\
For $(x,\xi_1,\dots,\xi_k)\in U_\a\times\E^k$,
define the curve $\bar{\mu}_k$ inductively as follows; $\bar{\mu}_1(t)=x+t\xi_1$, $\bar{\mu}_2(t)=\bar{\mu}_1(t)+\frac{t^2}{2}\{2\xi_2-\st{1}{M}_\a(x,\xi_1)\xi_1\}$ and for $i\geq 2$,
\begin{eqnarray*}
&{\bar{\mu}}_i(t)=\bar{\mu}_{i-1}(t)+\frac{t^i}{i}\{i\xi_i-\m_\a(x,\xi_1)\frac{\bar{\mu}^{(i-1)}_{i-1}(0)}{(i-2)!}
-\dots&\\
&\stackrel{i-1}{M}_\a(x,\xi_1,\frac{1}{2!}\bar{\mu}_{i-1}^{(2)}(0),...,\frac{1}{(i-1)!}\bar{\mu}_{i-1}^{(i-1)}(0))\xi_1\}.&
\end{eqnarray*}
%
%
%
In order to reduce the intricate computations as much as possible,  we will use the following lemma.
\begin{Lem}\label{Lem. d^k/ds d^k-1 t(f o c)(t,s)}
Fix the positive integer $k\geq 2$ and suppose that $\bar{\mu}:=\bar{\mu}_k$ be the  curve  defined above, $\mathcal{O}\subseteq \E$ be open  and $f:\mathcal{O}\to \E$ be any smooth map. Then
\begin{equation}\label{Eq. d^k/ds d^k-1 t(f o c)(t,s)}
\frac{\partial^{k}}{\partial s\partial t^{k-1}}(f\o \bar{d}_k)(t,s)|_{t=s=0}=(f\o \bar{\mu})^{(k)}(t)|_{t=0}
\end{equation}
where
\begin{equation*}
\bar{d}_k:(-\eps,\eps)^2\to \E ~;~~ (t,s)\mto \sum_{i=1}^{k-1}\frac{t^i}{i!}(\bar{\mu}^{(i)}(0)+s\bar{\mu}^{(i+1)}(0)).
\end{equation*}
\end{Lem}
\begin{proof}
See appendix.
\end{proof}
%
%
%
Now, suppose that ${\mu}:=\p_\a^{-1}\o\bar{\mu}$. Then  the local representation
of $T^kg$ is
\begin{eqnarray*}
\Psi_\b^k\o T^kg\o{\Phi_\a^k}^{-1}(x,\xi_1,\dots,\xi_k)=\Psi_\b^k\o T^kg([{\mu},x]_k)=\Psi_\b^k([g\o{\mu},g(x)]_k).
\end{eqnarray*}
With the  fact $\s_\b\o g\o{\mu}=\s_\b\o g\o\p_\a^{-1}\o\p_\a\o{\mu}:=g_{\b\a}\o\bar{\mu}$ in mind, we apply
the vector bundle trivialization of theorem \ref{osc k admits a vb} to $[g\o{\mu},g(x)]_k$ and we get
\begin{eqnarray*}
\Psi_\b^k([g\o{\mu},g(x)]_k)=\Big( g_{\b\a}(x),dg_{\b\a}(x)\xi_1,z_\b^1([{\mu}\o g,g(x)]_k),\dots,z_\b^k([{\mu}\o g,g(x)]_k) \Big)
\end{eqnarray*}
However,  for $2\leq i\leq k$ we have
\begin{eqnarray}\label{local form of T^kg primary}
&&iz_\b^i([g\o {\mu},g(x)]_k)\\
\nonumber&&= dg_{\b\a}(x)\Big[i\xi_i-\stackrel{1}{M}_\a(v_1)\frac{(\bar{\mu})^{(i-1)}(0)}{(i-2)!}-\dots-\stackrel{i-1}{M}_\a( v_{i-1})(\bar{\mu})^{(1)}(0)\Big] \\
\nonumber&&+\frac{1}{(i-1)!}\Big\{ \sum_{l_1+l_2=i}a_{(l_1,l_2)}^i d^2g_{\b\a}(x)\Big((\bar{\mu})^{(l_1)}(0),(\bar{\mu})^{(l_2)}(0)\Big)+\dots\\
\nonumber& & + d^ig_{\b\a}(x)\Big( (\bar{\mu}^{(1)})(0),\dots, (\bar{\mu}^{(1)})(0)\Big)\Big\}\\
\nonumber&&+\stackrel{1}{N}_\b\big(   \bar{v}_1   \big)\frac{(g_{\b\a}\o\bar{\mu})^{(i-1)}(0)}{(i-2)!}+\dots
+\stackrel{i-1}{N}_\b\big(  \bar{v}_{i-1}  \big)
(g_{\b\a}\o\bar{\mu})^{(1)}(0)
\end{eqnarray}
where $v=\big(\bar{\mu}(0),\bar{\mu}^{(1)}(0),\dots,\frac{\bar{\mu}^{k}(0)}{k!}\big) $, $v_j=Proj_{j+1}(v)$,
\begin{eqnarray*}
\bar{v}&=&\big((g_{\b\a}\o\bar{d})(0,0), \frac{\partial}{\partial t} (g_{\b\a}\o\bar{d})(0,0)
,\dots,\frac{\partial^{k}}{\partial t^{k}}\frac{(g_{\b\a}\o\bar{d})(0,0)}{k!}\big)\\
&=&\big((g_{\b\a}\o\bar{{\mu}})(0), (g_{\b\a}\o\bar{\mu})^{(1)}(0)
,\dots,\frac{(g_{\b\a}\o\bar{{\mu}})^{(k)}(0)}{k!}\big)
\end{eqnarray*}
and $\bar{v}_j=Proj_{j+1}(\bar{v})$.

As we can see now, due to the presence of higher order derivatives, as $k$ increases, (generally) it becomes increasingly difficult
for  $T^kg$ to be a vector bundle  morphism.

To get around this difficulty, let $K_M$ and $K_N$ be  connection maps which are induced by linear connections  $\nabla_M$ and $\nabla_N$ respectively.
%
%
\begin{The}\label{T^kg becomes a v.b morphism}
 If $\nabla_M$ and $ \nabla_N$ are $g$-related then, $T^kg:T^kM\to T^kN$ becomes a vector bundle  morphism.
\end{The}
\begin{proof}
If $[{\mu},x]_k\in T^k_xM$, then  $\pi_N^k\o T^kg ([{\mu},x]_k)=\pi_N^k([g\o{\mu},g(x)]_k)=g(x)=g\o\pi_M^k([{\mu},x]_k)$ means
that $T^kg$ is fibre preserving.   Now  consider the trivializations $({\pi_M^k}^{-1}(U_\a),\P_\a^k)$ and $({\pi_N^k}^{-1}(V_\b),\S_\b^k)$, as introduced in
theorem \ref{osc k admits a vb},  around $x$ and $g(x)$ respectively.
Our main task is to show that $\Psi_\b^k\o T^kg\o{\Phi_\a^k}^{-1}$ is linear on fibres.


\textbf{Step 1.} Setting $\e_i:=\frac{1}{(i-1)!}\bar{\mu}^{(i)}(0)$, $y:=\bar{\mu}^{(1)}(0)$ and $\xi_i:=\frac{1}{i!}\bar{\mu}^{(i)}(0)$ in  compatibility condition (\ref{g related local compatiblity condition}), and using lemma \ref{Lem. d^k/ds d^k-1 t(f o c)(t,s)} with $g_{\b\a}=f$ we get
\begin{eqnarray*}
&dg_{\b\a}(x)\frac{\bar{\mu}^{(i)}(0)}{(i-1)!} +  dg_{\b\a}(x)\st{1}{M}_\a(v_1) \frac{\bar{\mu}^{(i-1)}(0)}{(i-2)!}
\dots+  dg_{\b\a}(x) \st{i-1}{M}_\a(v_{i-1})\bar{\mu}^{(1)}(0)&
\\
&=\frac{\partial^i}{\partial s\partial t^{i-1}}\frac{(g_{\b\a}\o\bar{d}_k)(0,0)}{(i-1)!}+\st{1}{N}_\b(\bar{v}_1)
\frac{\partial^{i-1}}{\partial s\partial t^{i-2}}\frac{(g_{\b\a}\o\bar{d}_k)(0,0)}{(i-2)!}+&\\
&\dots+\st{i-1}{N}_\b(\bar{v}_{i-1})\frac{\partial}{\partial s}(g_{\b\a}\o\bar{d}_k)(0,0)&\\
&=\frac{(g_{\b\a}\o\bar{\mu})^{(i)}(0)}{(i-1)!}+\st{1}{N}_\b(\bar{v}_1)\frac{(g_{\b\a}\o\bar{\mu})^{(i-1)}(0)}{(i-2)!}+
\dots+\st{i-1}{N}_\b(\bar{v}_{i-1})(g_{\b\a}\o\bar{\mu})^{(1)}(0).&
\end{eqnarray*}
As a consequence we obtain
\begin{eqnarray}\label{step 1}
&&
\st{1}{N}_\b(\bar{v}_1)\frac{(g_{\b\a}\o\bar{\mu})^{(i-1)}(0)}{(i-2)!}+
\dots\st{i-1}{N}_\b(\bar{v}_{i-1})(g_{\b\a}\o\bar{\mu})^{(1)}(0)\\
\nonumber&&=dg_{\b\a}(x)\frac{\bar{\mu}^{(i)}(0)}{(i-1)!} +  dg_{\b\a}(x)\st{1}{M}_\a(v_1) \frac{\bar{\mu}^{(i-1)}(0)}{(i-2)!}+
\dots\\
\nonumber&&\hspace{5mm}+  dg_{\b\a}(x) \st{i-1}{M}_\a(v_{i-1})\bar{\mu}^{(1)}(0)-\frac{(g_{\b\a}\o\bar{\mu})^{(i)}(0)}{(i-1)!}.
\end{eqnarray}
%
%
\textbf{Step 2.}
We now  apply the previous observation  to (\ref{local form of T^kg primary}) and we get
\begin{eqnarray*}
&i z_\b^k([g\o{\mu},g(x)]_k)&   \\
&=idg_{\b\a}(x)\xi_i+\frac{1}{(i-1)!}\Big\{ \sum_{l_1+l_2=i}a_{(l_1,l_2)}^i d^2g_{\b\a}(x)\Big((\bar{\mu})^{(l_1)}(0),(\bar{\mu})^{(l_2)}(0)\Big)&\\
& \hspace{5mm}+\dots+ d^ig_{\b\a}(x)\Big( (\bar{\mu}^{(1)})(0),\dots, (\bar{\mu}^{(1)})(0)\Big)\Big\}+dg_{\b\a}(x)\frac{\bar{\mu}^{(i)}(0)}{(i-1)!}&\\
&\hspace{5mm}-\frac{(g_{\b\a}\o\bar{\mu})^{(i)}(0)}{(i-1)!}&\\
&=idg_{\b\a}(x)\xi_i+\frac{1}{(i-1)!}\Big\{ \sum_{l_1+l_2=i}a_{(l_1,l_2)}^i d^2g_{\b\a}(x)\Big((\bar{\mu})^{(l_1)}(0),(\bar{\mu})^{(l_2)}(0)\Big)&\\
& \hspace{5mm}+\dots+ d^ig_{\b\a}(x)\Big( (\bar{\mu}^{(1)})(0),\dots, (\bar{\mu}^{(1)})(0)\Big)\Big\}+dg_{\b\a}(x)\frac{\bar{\mu}^{(i)}(0)}{(i-1)!}&\\
&\hspace{5mm}-\frac{1}{(i-1)!}\Big\{dg_{\b\a}(x)\bar{\mu}^{(i)}(0)+\sum_{l_1+l_2=i}a_{(l_1,l_2)}^i d^2g_{\b\a}(x)\Big((\bar{\mu})^{(l_1)}(0)&\\
&\hspace{5mm},(\bar{\mu})^{(l_2)}(0)\Big)+\dots+ d^ig_{\b\a}(x)\Big( (\bar{\mu}^{(1)})(0),\dots, (\bar{\mu}^{(1)})(0)\Big)\Big\}&\\
&=idg_{\b\a}(x)\xi_i.&
\end{eqnarray*}
As a consequence we have
\begin{eqnarray*}
\Psi_\b^k\o T^kg\o{\Phi_\a^k}^{-1}(x,\xi_1,\dots,\xi_k)
%
%
=\Big(g_{\b\a}(x),dg_{\b\a}(x)(\xi_1),\dots, dg_{\b\a}(x)(\xi_k) \Big).
\end{eqnarray*}
This last means that $T^kg$ is fibre linear and
\begin{eqnarray*}
T^kg_{\b\a}: U_\a&\to & {\mathcal{L}}(\E^k,\E'^k)\\
x &\mto& \S_{\b,g(x)}^k\o T^k_xg\o {\P_{\a,x}^k}^{-1}
\end{eqnarray*}
is a smooth morphism which completes the proof.
\end{proof}
%
%
%
\begin{Rem}
One can replace the induced $g$-related connection maps with general $g$-related connection maps and prove  theorem \ref{T^kg becomes a v.b morphism} in an exactly similar way.
\end{Rem}
\begin{Rem}
If $M=N$ and $g=id_M$ then  theorem \ref{T^kg becomes a v.b morphism} yields
\begin{eqnarray*}
\P_\b^k\o{\P_\a^k}^{-1}:U_\a&\to & {\mathcal{L}}(\E^k,\E^k)\\
x &\mto& (d\p_\b\o\p_\a^{-1}(x)(.),\dots,d\p_\b\o\p_\a^{-1}(x)(.))
\end{eqnarray*}
as it was noted  by  theorem 3.3 of \cite{Suri Osck}.
\end{Rem}
The next corollary is a direct consequence of theorems \ref{T^kg becomes a v.b morphism} and \ref{theorem lifts of related connections remain related}.
\begin{Cor}
If $g$ is a diffeomorphism, then $T^kg$ becomes a vector bundle isomorphism.
\end{Cor}
\begin{Rem}
Let $\nabla_M$ and $\nabla_N$ be two $g$-related connections on $M$ and $N$ respectively. Consider the vector bundle structures on
$(T^kM,\pi_M^{k,i},T^iM)$ and $(T^kN,\pi_N^{k,i},T^iN)$ proposed by theorem \ref{osc k admits a vb}. Then the argument of theorem \ref{T^kg becomes a v.b morphism} can easily be modified to prove
\begin{eqnarray*}
(T^kg,T^ig):(T^kM,\pi_M^{k,i},T^iM)\to (T^kN,\pi_N^{k,i},T^iN)
\end{eqnarray*}
is also  a vector bundle morphism for $i<k$.
\end{Rem}
%
%
%
%
%
%
\section{$T^\infty g$ as a vector bundle morphism}\label{section T infty g}
As we have seen in the previous sections, at the presence of $g$-related connections on $M$ and $N$ for any $k\in \N$, $(T^kg,g)$ becomes a vector bundle morphism. If we take one step further by considering $T^\infty M$ and $T^\infty N$, as generalized Fr\'{e}chet vector bundles over $M$ and $N$ respectively (\cite{Suri Osck}), then proving $T^\infty g$ to be a generalized vector bundle morphism becomes much more complicated.

More precisely the set of linear maps between $\F:=\lim \E^k$ and $\F^\prime:=\lim {\E'}^k$ (the fibre types  of $(T^\infty M,\pi^\infty_M, M)$ and $(T^\infty N,\pi^\infty_N, N)$ respectively) does not remain in the category of Fr\'{e}chet spaces \cite{Hamilton, Sch}.

In this section employing the projective limit methodology, as in \cite{split, GAl-VB, Suri Osck, Suri Hfbk} etc., we show that at the presence of $g$-related connections on $M$ and $N$, $(T^\infty g,g)$  becomes a  vector bundle morphism.

Of course one can consider a projective system  of  $g$-related connection maps on $T^kM$ and $T^kN$, $k\in \N$ and prove the same results.

Let the notation be as in the preceding sections and for the \textbf{natural numbers $j\geq i$} consider the  projections $\pi^{j,i}_M:T^jM\to T^iM$ and $\pi^{j,i}_N:T^jN\to T^iN$ mapping $[\g,x]_j$ onto $[\g,x]_i$ as connecting morphisms of the projective families $\{T^kM\}_{k\in \N}$ and $\{T^kN\}_{k\in \N}$ (for more details see \cite{Suri Osck}, \cite{Suri Hfbk}. The family $\{T^kg\}_{k\in \N}$ form a projective system of maps since $\pi^{j,i}_N\o T^jg=T^ig\o\pi^{j,i}_M$. More precisely
\begin{eqnarray*}
\pi^{j,i}_N\o T^jg([\g,x]_j)=\pi^{j,i}_N ([g\o\g,g(x)]_j)=[g\o\g,g(x)]_i
\end{eqnarray*}
and
\begin{eqnarray*}
T^ig\o\pi^{j,i}_M([\g,x]_j)=T^ig([\g,x]_i)=[g\o\g,g(x)]_i.
\end{eqnarray*}
As a consequence the limit map $T^\infty g:=\lim T^kg$ exists and maps the thread $([\g,x]_k)_{k\in \N}\in T^\infty M:=\lim T^kM$ to $([g\o\g,x]_k)_{k\in \N}\in T^\infty N=\lim T^kN$.
It is  easily checked that, for any $x\in  M$,  the families    $\{\P_{\a,x}^k\}_{k\in\N}$ and $\{\S_{\b,g(x)}^k\}_{k\in\N}$, as in theorem \ref{osc k admits a vb}, form projective systems of trivializations with the limits  ${\P_{\a,x}^\infty}$ and $\S_{\b,g(x)}^\infty$ respectively (see also \cite{GAl-VB}).

But $T^\infty g$ seems to be far from being called a vector bundle morphism due to the
difficulties emerged in $\mathcal{L}(\F,\F')$ and therefore the problematic map
\begin{eqnarray*}
T^\infty g_{\b\a}:U_\a&\to& \mathcal{L}(\F,\F')\\
x&\mto&\S_{\b,g(x)}^\infty\o T^\infty\o g\o{\P_{\a,x}^\infty}^{-1}.
\end{eqnarray*}
To overcome this obstacle define
\begin{equation*}
\mathcal{H}(\F,\F')=\{(l_k)_{k\in\N}\in\prod_{k=1}^\infty\mathcal{L}(\E^k,\E'^k):~\r_{ji}'\o l_j=l_i\o\r_{ji},~\forall ~j\geq i\}
\end{equation*}
where  $\r_{ji}:\E^j\to \E^i$ and $\r_{ji}':{\E'}^j\to{\E'}^i$ are the projection maps to the first $i$ factors. $\mathcal{H}(\F,\F')$ is a   Fr\'{e}chet space  (\cite{GAl-VB}) isomorphic
to the projective limit of the projective system of Banach spaces $\{\mathcal{H}^k(\E^k,\E'^k)\}_{k\in\N}$ where
\begin{equation*}
\mathcal{H}^k(\E^k,\E'^k)=\{(l_i)_{1\leq i\leq k}\in\prod_{i=1}^k\mathcal{L}(\E^i,\E'^i):~\r_{ji}'\o l_j=l_i\o\r_{ji},~\forall ~1\leq i\leq j\leq k\}.
\end{equation*}
However, for any $\xi_1,\dots,\xi_j\in \E$ and $j\geq i$ we have
\begin{eqnarray*}
&&\r_{ji}'\o(\S_{\b,g(x)}^j\o T^j g\o{\P_{\a,x}^j}^{-1})(\xi_1,\dots,\xi_j)=\r_{ji}'(dg_{\b\a}(x)\xi_1,\dots,dg_{\b\a}(x)\xi_j)\\
&&=(dg_{\b\a}(x)\xi_1,\dots,dg_{\b\a}(x)\xi_i)
=(\S_{\b,g(x)}^i\o T^i g\o{\P_{\a,x}^i}^{-1})\o\r_{ji}(\xi_1,\dots,\xi_j)
\end{eqnarray*}
meaning that
$\{\S_{\b,g(x)}^k\o T^k g\o{\P_{\a,x}^k}^{-1}\}_{k\in\N}$, with the limit $\S_{\b,g(x)}^\infty\o T^\infty g\o{\P_{\a,x}^\infty}^{-1}$, belongs to $\mathcal{H}(\F,\F')$.
As a consequence $T^{\infty } g_{\b\a}:U_\a\to \mathcal{L}(\F,\F')$; $x\mto\varepsilon\o \S_{\b,g(x)}^\infty\o T^\infty g\o{\P_{\a,x}^\infty}^{-1}  $ is smooth (in the sense of Leslie and Galanis \cite{Leslie1, Leslie2, GAL-PLB, GAl-VB}) where $\varepsilon$ is the linear (and smooth) map defined by
\begin{eqnarray*}
\varepsilon:\mathcal{H}(\F,\F')&\to &\mathcal{L}(\F,\F')\\
(l_k)_{k\in\N}&\mto&\lim l_k.
\end{eqnarray*}
As a result, $T^\infty g$ is a generalized   smooth map.
%
%
%
\begin{Rem}
It is known that differential calculus in normed spaces does not have a unique canonical extension to general topological spaces (For a list of definitions see e.g. \cite{Averbuh 1, Keller}).

In our case, the problem is due to $\mathcal{L}(\F,\F')$ which is not normable. In fact $\mathcal{L}(\F,\F')$  drops out of the category of Fr\'{e}chet spaces and consequently it can not be considered as a projective limit of Banach spaces.

Considering the Fr\'{e}chet space $\mathcal{H}(\F,\F')$ allows  us to consider $\S_{\b,g(x)}^\infty\o T^\infty g\o{\P_{\a,x}^\infty}^{-1} =\lim \S_{\b,g(x)}^k\o T^k g\o{\P_{\a,x}^k}^{-1}$ as a map with values in $\mathcal{H}(\F,\F')=\lim \mathcal{H}^k(\E^k,\E'^k)$. Considering  the auxiliary Banach spaces
$\{\mathcal{H}^k(\E^k,\E'^k)\}_{k\in \N}$, the map   $\lim \S_{\b,g(x)}^k\o T^k g\o{\P_{\a,x}^k}^{-1}$ is called a (generalized) smooth map since its factors are smooth \cite{Leslie1, Leslie2, GAl-VB}.
\end{Rem}
%
%
%
Finally, following the argument of section \ref{section Tkg as a vb morphism} we get:
\begin{The}\label{Tinfty g becomes a vb morphism}
The pair $(T^\infty g,g):(T^\infty M,\pi_M^\infty,M)\to (T^\infty N,\pi_N^\infty,N)$ is a generalized vector bundle morphism. Moreover, the bundle morphism $(T^\infty g,g)$ is a vector bundle isomorphism if   $g$ is a diffeomorphism.
\end{The}
\begin{Rem}
It is easy to check that $(T^\infty g,T^kg):(T^\infty M,\pi_M^{\infty,k},T^kM)\to (T^\infty N,\pi_N^{\infty,k},T^kN)$, $1\leq k<\infty$, is also a vector bundle morphism where $\pi_M^{\infty,k}$ maps the class $[\g,x]_\infty$ to $[\g,x]_k$.
\end{Rem}

%
%
\section{Applications and examples}
In this section first, we  settle our results to the special case of $f:M\to N$, where $f$ is an immersion and $N$ is a Riemannian Hilbert manifold.
Then we study the vector bundle dependence on convex combination of connection maps and finally the case of
manifold of $C^r$ maps  between manifolds $M$ and $N$ is considered.

\begin{Examp}
Let $M$ and $N$ be two smooth manifolds modeled on the Hilbert spaces $\E$ and $\F$ respectively and $f:M\to N$ be a smooth immersion.
Moreover suppose that $h$ be a Riemannian metric on $N$ with the Levi-Civita connection $\nabla_N$. It is known that the immersion $f$  induces a Riemannian metric $g$ on $M$ defined by
$$g(p)(u,v):=h(f(p))(d_pfu,d_pfv);~~\forall p\in M ~~\textrm{and} ~\forall u,v\in T_pM.$$
Denote by $\nabla_M$ the associated Levi-Civita connection of $g$. In what follows, we will show that $\nabla_M$ and $\nabla_N$ are $f$-related.

For $p\in M$ consider the charts $(U_\a,\phi_\a)$ and $(V_\b,\s_\b)$ around $p$ and $f(p):=q$ respectively. Setting
\begin{equation*}
g_\a(p_0):=g(p)\circ(d_{p_0}\p_\a^{-1}\times d_{p_0}\p_\a^{-1}):\E\times \E\to \R
\end{equation*}
and
\begin{equation*}
h_\b({q_0}):=h(q)\circ(d_{q_0}\s_\b^{-1}\times d_{q_0}\s_\b^{-1}):\F\times \F\to \R
\end{equation*}
we observe that
\begin{eqnarray}\label{local rep pf pull backmetric}
\nonumber g_\a(p_0)(u,v)&=&g(p)(d_{p_0}\p_\a^{-1}u, d_{p_0}\p_\a^{-1}v)\\
&=&h(f(p))(d_{p_0}f\o\p_\a^{-1}u, d_{p_0}f\o\p_\a^{-1}v)\\
%
%
\nonumber&=&h_\b(q_0)(d_{p_0}f_{\b\a}u,d_{p_0}f_{\b\a}v)
\end{eqnarray}
where $p_0:=\p(p)$, $q_0:=f_{\b\a}( p_0)$ and $f_{\b\a}:=\s_\b\o f\o\p_\a^{-1}$.
As a consequence of (\ref{local rep pf pull backmetric}) and  Leibniz's   rule we have
\begin{eqnarray*}
&& dg_\a(p_0).w(u,v)=dh_\b( q_0).d_{p_0}f_{\b\a}w\Big(d_{p_0}f_{\b\a}u,d_{p_0}f_{\b\a}v\Big)\\
&&+h_\b(q_0)\Big(d^2_{p_0}f_{\b\a}(w,u),d_{p_0}f_{\b\a}v\Big) +h_\b(q_0)\Big(d_{p_0}f_{\b\a}u,d_{p_0}^2f_{\b\a}(w,v)\Big)
\end{eqnarray*}
Now, using the  Koszul formula   (e.g. \cite{Flaschel, Kling}) we get
\begin{eqnarray*}
&g_\a(p_0)(\G_\a^M(p_0)(u,v),w)=\frac{1}{2}\{    dg_\a(p_0).u(v,w)  +dg_\a(p_0).v(u,w)&\\
&\hspace{40mm}-dg_\a(p_0).w(u,v)  \}
\end{eqnarray*}
\begin{eqnarray*}
&=\frac{1}{2}\Big\{  dh_\b( q_0).d_{p_0}f_{\b\a}u\Big(d_{p_0}f_{\b\a}v,d_{p_0}f_{\b\a}w\Big)     +   h_\b(q_0)\Big(d^2_{p_0}f_{\b\a}(u,v),d_{p_0}f_{\b\a}w\Big)&\\
&\hspace{40mm}+h_\b(q_0)\Big(d_{p_0}f_{\b\a}v,d_{p_0}^2f_{\b\a}(u,w)\Big)&
\end{eqnarray*}
\begin{eqnarray*}
&+dh_\b( q_0).d_{p_0}f_{\b\a}v\Big(d_{p_0}f_{\b\a}u,d_{p_0}f_{\b\a}w\Big)    +   h_\b(q_0)\Big(d^2_{p_0}f_{\b\a}(v,u),d_{p_0}f_{\b\a}w\Big)&\\
&\hspace{40mm}+h_\b(q_0)\Big(d_{p_0}f_{\b\a}u,d_{p_0}^2f_{\b\a}(v,w)\Big)&\\
&- dh_\b( q_0).d_{p_0}f_{\b\a}w\Big(d_{p_0}f_{\b\a}u,d_{p_0}f_{\b\a}v\Big)    -     h_\b(q_0)\Big(d^2_{p_0}f_{\b\a}(w,u),d_{p_0}f_{\b\a}v\Big)&\\
&\hspace{40mm}-h_\b(q_0)\Big(d_{p_0}f_{\b\a}u,d_{p_0}^2f_{\b\a}(w,v) \Big)\Big\}&
\end{eqnarray*}
\begin{eqnarray*}
&\hspace{-5mm}=h_\b(q_0)\Big(\G_\b^N(q_0)(d_{p_0}f_{\b\a}u,d_{p_0}f_{\b\a}v),d_{p_0}f_{\b\a}w\Big)   +    h_\b(q_0)\Big(d^2_{p_0}f_{\b\a}(u,v),d_{p_0}f_{\b\a}w\Big)&
\end{eqnarray*}
for any $u,v,w\in \E$. On the other hand equation (\ref{local rep pf pull backmetric}) now reads
\begin{eqnarray*}
g_\a(p_0)(\G_\a^M(p_0)(u,v),w)=h_\b(q_0)(d_{p_0}f_{\b\a}\G_\a^M(p_0)(u,v),d_{p_0}f_{\b\a}w).
\end{eqnarray*}
Since $g$ is  non-degenerate  we deduce that
\begin{eqnarray*}
d_{p_0}f_{\b\a} \G_\a^M(p_0)(u,v)=   \G_\b^N(q_0)(d_{p_0}f_{\b\a}u,d_{p_0}f_{\b\a}v)+d^2_{p_0}f_{\b\a}(u,v).
\end{eqnarray*}
However, the last equality is nothing but the local compatibility condition, described by remark \ref{Rem compatibility condition for connections when k=1}, for the $f$-related connections $\nabla_M$ and $\nabla_N$.

As a consequence of theorems \ref{osc k admits a vb}, \ref{T^kg becomes a v.b morphism} and \ref{Tinfty g becomes a vb morphism} for $k\in \N\cup\{\infty\}$, $T^kM$ and $T^kN$ admit vector bundle structures and  in this case $(T^kf,f):$ $(T^kM,\pi_M^k,M)$ $\to (T^kN,\pi_N^k,N)$  becomes  a vector bundle morphism.
Moreover if $f$ is an diffeomorphism (isometry) then, $T^kf$ is a vector bundle isomorphism. In this case, for $k\in\N$, with respect to the induced metrics introduced in section 3.3 of \cite{Suri Osck}, $T^kg$ is  also an isometry.
\end{Examp}
%
%
%
%
%
%
\begin{Examp}\label{Example convex combination}
Let $K=(\K,\dots,\KKK)$ and $\bar{K}=(\stackrel{1}{\bar K},\dots,\stackrel{k}{\bar K})$ be two connection maps on $T^kM$ and $0\leq \l\leq 1$.
We note that for any $1\leq a\leq k-1$,
\begin{eqnarray*}
(\l K+(1-\l)\bar K)^k\o \mathbb{J}^a   & =&    \l\KKK\o \mathbb{J}^a+(1-\l)\stackrel{k}{\bar K}\o \mathbb{J}^a\\
&=&   \l\stackrel{k-a}{ K} +(1-\l)\stackrel{k-a}{\bar K}\\
&=&    (\l K+(1-\l) \bar{K})^{k-a}
\end{eqnarray*}
and
\begin{eqnarray*}
(\l K+(1-\l)\bar K)^k\o \mathbb{J}^k   & =&    \l\KKK\o \mathbb{J}^k+(1-\l)\stackrel{k}{\bar K}\o \mathbb{J}^k\\
&=&   \l{\pi^k_M}_* +(1-\l){\pi^k_M}_*={\pi^k_M}_*
\end{eqnarray*}
that is the convex combination of $K$ and $\bar K$
$$\l K+(1-\l) \bar K =\Big(  \l \K+(1-\l)\stackrel{1}{\bar K},\dots, \l \KKK+(1-\l)\stackrel{k}{\bar K}   \Big)$$
is a connection map too. Moreover it is easily seen that for any $u=(x,\xi_1,\dots,\xi_k)$ and $(u,y,\e_1,\dots,\e_k)\in T_uT^kM$
\begin{eqnarray*}
&&(\l K+(1-\l)\bar K)|_{U_\a}(u;y,\e_1,...,\e_k)=\bigoplus_{i=1}^k\Big(x, \e_i+ \\
&&   (   \l\m_\a(u) +(1-\l)\stackrel{1}{\bar M}_\a(u)   )\e_{i-1}+
\dots+   (\l \stackrel{i}{M}_\a(u) +(1-\l) \stackrel{i}{\bar M}_\a(u)   ) y\Big).
\end{eqnarray*}
As a consequence of theorem \ref{osc k admits a vb}, for any $x\in M$
\begin{eqnarray*}
(T^kM_{\l K+(1-\l)\bar K})_x&=&(T^kM_{\l K})_x+(T^kM_{(1-\l)\bar K})_x\\
&=&  \l(T^kM_{ K})_x+(1-\l)(T^kM_{\bar K})_x
\end{eqnarray*}
where $ T^kM_K$, $T^kM_{\bar K}$ and $T^kM_{\l K+(1-\l)\bar K}$ denote the vector bundle structures on $\pi_M^k:T^kM\to M$ induced by the
connection maps $K$, $\bar K$ and $\l K+(1-\l)\bar K$ respectively.

This last result means that the fibres of the vector bundle structure induced by the convex combination $\l K+(1-\l)\bar K$,
is the convex combination of fibres of the vector bundle structures induced by $K$ and $\bar K$.

A similar argument holds for any  convex combination $\sum_{i=1}^n\l_i K_i$ of connection maps.
\end{Examp}
%
%
%
%
\begin{Examp}
Let $N$ be a $C^\infty$ compact manifold and $M$ be a $C^\infty $ Banach (possibly infinite dimensional) manifold with a linear connection $\nabla_M$.  According to \cite{Eli} $C^r(N,M)$, the space of all $C^r$ maps $0\leq r<\infty$ from $N$ to $M$, forms a Banach manifold with the following charts. Let $exp:\mathcal{O}\subset TM\to M$ be the exponential map corresponding to the linear connection $\nabla_M$. Moreover suppose that $\mathcal{D}$ be an open neighborhood of the zero section in $TM$ such that $(\pi_M^1,exp)|_\mathcal{D}$ form   $\mathcal{D}$ to $(\pi_M^1,exp)(\mathcal{D})\subset M\times M$ is a diffeomorphism.   For the $C^r$ map $h:N\to M$ the chart $(\p_h,U_h)$ defined by
\begin{equation*}
C^r(exp):C^r(h^*\mathcal{D})  \to  C^r(N,M) ~;~ \xi \mto exp\o \xi.
\end{equation*}
In the notation above $C^r(h^*\mathcal{D})$ is the set of all sections $\xi:N\to TM$ with the property $\pi_M^1\o\xi=h$. In this case $C^r(h^*\mathcal{D})$ becomes a Banach space with the norm
$$\|\xi\|_{C^r}=\sum_{j=0}^r\|\nabla^j\xi\|_{C^0}:=\sum_{j=0}^r sup\|\nabla^j\xi(p)\|_{p\in N}.$$
which serves as the model space of $C^r(N,M)$.

Moreover the connection $\nabla_M$ induces a connection on $C^r(N,M)$ with the connection map $C^r(\nabla_M):TTC^r(N,M)\simeq C^r(N,TTM)\to C^r(N,TM)\simeq TC^r(N,M)$ which maps $A\in C^r(N,TTM)$ to $C^r(\nabla_M)(A)=\nabla_M\o A$. Since $\nabla_M$ is a linear connection, so $C^r(\nabla_M)$ also is a linear connection (\cite{Eli} Theorem 5.4).

According to theorem \ref{osc k admits a vb}, $T^kM$ and $T^kC^r(N,M)$ admits vector bundle structures on $M$ and $C^r(N,M)$ respectively. Moreover
$T^kC^r(N,M)\simeq \oplus_{j=1}^kTC^r(N,M)$ and $C^r(N,T^kM)\simeq C^r(N,\oplus_{j=1}^kTM)$ are isomorphic vector bundles over $C^r(N,M)$.

For the Banach manifolds $M$ and $M'$ and the smooth map $g:M\to M'$ the map $C^r(g):C^r(N,M)\to C^r(N,M')$ defined by $f\mto g\o f$ is differentiable and $TC^r(g)=C^r(Tg)$ \cite{Eli}.

Now, suppose that $\nabla_M$ and $\nabla_{M'}$ be two $g$-related connections on $M$ and $M'$ respectively. Then
\begin{eqnarray*}
C^r(\nabla_{M'})\o TTC^r(g)&=&C^r(\nabla_{M'})\o C^r(TTg)=C^r(\nabla_{M'}\o TTg)\\
&=&C^r(Tg\o\nabla_M)=TC^r(g)\o C^r(\nabla_M)
\end{eqnarray*}
that is $C^r(\nabla_M)$ and $C^r(\nabla_{M'})$ are $C^r(g)$-related. As a consequence of theorem \ref{T^kg becomes a v.b morphism}, $T^kg:T^kM\to T^kM'$ and $T^kC^r(g):T^kC^r(N,M)\to T^kC^r(N,M')$ are vector bundle morphisms.
\end{Examp}
%
%
%
%
\section{Appendix}
In this section,  using the chain rule formula (\ref{higher order chain rule}) we  prove  lemma \ref{Lem. d^k/ds d^k-1 t(f o c)(t,s)}.
For $(x,\xi_1,\dots,\xi_k)\in U_\a\times\E^k$,
define the curve $\bar{\mu}_k$ inductively as in section \ref{section Tkg as a vb morphism} by $\bar{\mu}_1(t)=x+t\xi_1$, $\bar{\mu}_2(t)=\bar{\mu}_1(t)+\frac{t^2}{2}\{2\xi_2-\st{1}{M}_\a(x,\xi_1)\xi_1\}$ and for $i\geq 2$,
\begin{eqnarray*}
&{\bar{\mu}}_i(t)=\bar{\mu}_{i-1}(t)+\frac{t^i}{i}\{i\xi_i-\m_\a(x,\xi_1)\frac{\bar{\mu}^{(i-1)}_{i-1}(0)}{(i-2)!}
-\dots&\\
&\stackrel{i-1}{M}_\a(x,\xi_1,\frac{1}{2!}\bar{\mu}_{i-1}^{(2)}(0),...,\frac{1}{(i-1)!}\bar{\mu}_{i-1}^{(i-1)}(0))\xi_1\}.&
\end{eqnarray*}
\begin{Lem}
Let $\bar{\mu}:=\bar{\mu}_k$ be the  map  defined above, $\mathcal{O}\subseteq \E$ be open  and $f:\mathcal{O}\to \E$ be any smooth map. Then
\begin{equation}\label{Eq. d^k/ds d^k-1 t(f o c)(t,s)}
\frac{\partial^{k}}{\partial s\partial t^{k-1}}(f\o \bar{d}_k)(t,s)|_{t=s=0}=(f\o \bar{\mu})^{(k)}(t)|_{t=0}
\end{equation}
where
\begin{equation*}
\bar{d}_k:(-\eps,\eps)^2\to \E ~;~~ (t,s)\mto \sum_{i=1}^{k-1}\frac{t^i}{i!}(\bar{\mu}^{(i)}(0)+s\bar{\mu}^{(i+1)}(0)).
\end{equation*}
\end{Lem}
\begin{proof}
Using the chain rule formula (\ref{higher order chain rule}) we observe that
\begin{eqnarray*}
&& \frac{\partial^{k}}{\partial s\partial t^{k-1}}(f\o \bar{d}_k)(t,s)|_{t=s=0}\\
&=&\frac{\partial}{\partial s}|_{s=0}
\{d f(x+s\xi_1)[\bar{\mu}^{(k-1)}(0)+ s \bar{\mu}^{(k)}(0)]\\
&&+ \sum_{l_1+l_2=k-1} a^{k-1}_{l_1,l_2} d^2f (x+s\xi_1)[\bar{\mu}^{(l_1)}(0)+s\bar{\mu}^{(l_1+1)}(0),\bar{\mu}^{(l_2)}(0)+s\bar{\mu}^{(l_2+1)}(0)]+\\
&&\dots+d^{k-1}(x+s\xi_1)[\bar{\mu}^{(1)}(0)+s\bar{\mu}^{(2)}(0),\dots,\bar{\mu}^{(1)}(0)+s\bar{\mu}^{(2)}(0)]\}\\
&=&d f(x)[\bar{\mu}^{(k)}(0)] +d^2f(x)[\bar{\mu}^{(1)}(0),\bar{\mu}^{(k-1)}(0)]+ (k-1)\{d^2f(x)[\bar{\mu}^{(2)}(0),\\
&&\bar{\mu}^{(k-2)}(0)]
+d^2f(x)[\bar{\mu}^{(1)}(0),\bar{\mu}^{(k-1)}(0)] + d^3f(x)[\bar{\mu}^{(1)}(0),\bar{\mu}^{(1)}(0),
\\
&&\bar{\mu}^{(k-2)}(0)]\} +\dots
+(k-1)d^{k-1}f(x)[\bar{\mu}^{(1)}(0),\dots,\bar{\mu}^{(1)}(0),\bar{\mu}^{(2)}(0)]\\
&&+d^kf(x)[\bar{\mu}^{(1)}(0),\dots,\bar{\mu}^{(1)}(0)]\\
&=& d f(x)[\bar{\mu}^{(k)}(0)]+ kd^2f(x)[\bar{\mu}^{(1)}(0),\bar{\mu}^{(k-1)}(0)]+\frac{k(k-1)}{2}d^2f(x)[\bar{\mu}^{(2)}(0),\\
&&\bar{\mu}^{(k-2)}(0)]+\dots+d^kf(x)[\bar{\mu}^{(1)}(0),\dots,\bar{\mu}^{(1)}(0)]\\
&=&(f\o \bar{\mu})^{(k)}(0)
\end{eqnarray*}
as we claimed.
\end{proof}
%
%
Finally, we leave it to the reader to verify that
\begin{Lem}\label{Lem 2}
For any differentiable function $f:O\subseteq \E\to \E$ and the maps $\bar{c}_1(t,s,h):=x+h\xi_1+s y+t\xi_1+\dots +t^k\xi_k$, $\bar\g_1(t,h):=\bar{c}_1(t,0,h)$, $\dots$,
$$\bar{c}_i(t,s,h):=x+sy +t\xi_1+\dots +t^{i-2}\xi_{i-2}+t^{i-1}(\xi_{i-1}+hi\xi_i)+t^i\xi_i+\dots +t^k\xi_k,$$
$\bar\g_i(t,h):=\bar{c}_i(t,0,h)$, $1\leq i\leq k$, and $\bar{c}(t,s)=x+sy+t\xi_1+\dots +t^k\xi_k$  the following properties hold.\\
\textbf{i.}
\begin{eqnarray*}
\frac{\partial^{j+1}}{\partial h \partial s \partial t^{j-1}}\sum_{i=1}^k (f\o \bar{c}_i)(t,s,h)|_{t=s=h=0}=\frac{\partial^{j+1}}{\partial
s\partial t^{j}} (f\o \bar{c})(t,s)|_{t=s=0}
\end{eqnarray*}
\textbf{ii.}
\begin{eqnarray*}
\frac{\partial^{j}}{\partial h  \partial t^{j-1}}\sum_{i=1}^k (f\o \bar{c}_i)(t,s,h)|_{t=s=h=0}=\frac{\partial^{j}}{\partial t^{j}} (f\o \bar{c})(t,s)|_{t=s=0}
\end{eqnarray*}
$j=1,\dots,k$.
\end{Lem}
%

%
%

\end{document}